\documentclass[12pt]{article}

\usepackage{algorithm,algorithmic}
\usepackage{bm}
\usepackage{tikz}
\usepackage{mathtools}
\usepackage{hyperref}
\usepackage{float}
\usepackage{amsmath}
\usepackage{amsthm}
\usepackage{amsbsy}
\usepackage{amsfonts}
\usepackage{fullpage}
\usepackage{times}

\DeclarePairedDelimiterX{\norm}[1]{\lVert}{\rVert}{#1}
\DeclarePairedDelimiterX{\normtto}[1]{\lVert}{\rVert_{2\rightarrow 1}}{#1}
\DeclarePairedDelimiterX{\abs}[1]{\lvert}{\rvert}{#1}
\DeclarePairedDelimiterX{\inp}[2]{\langle}{\rangle}{#1, #2}

\DeclareMathOperator*{\argmin}{arg\,min}

\newcommand{\mb}[1]{\mathbb{#1}}

\newcommand{\spccc}{\textbf{SPCCC}}
\newcommand{\lpccc}{\textbf{LP-CCC}}
\renewcommand{\S}{S}
\newcommand{\Sp}{S^\prime}
\newcommand{\fc}{FindComp}
\newcommand{\pg}{\text{Partition Graph}}
\newcommand{\giid}{\text{Generate IID Samples}}
\newcommand{\itest}{\text{Identity Test}}
\newcommand{\chg}{\chi}
\newcommand{\lprp}[1]{\left(#1\right)}
\newcommand{\lbrb}[1]{\left\{#1\right\}}
\newcommand{\ot}{\widetilde{O}}
\newcommand{\omt}{\widetilde{\Omega}}
\newcommand{\hitt}{\text{HitT}}
\newcommand{\dmc}{\text{Dist}}
\newcommand{\dhel}{d_{Hel}}
\newcommand{\dtv}{d_{TV}}
\newcommand{\distt}{\text{Dist}}
\newcommand{\nue}{\eta}
\newcommand{\sq}{\text{Sq}}

\newtheorem{claim}{Claim}
\newtheorem{question}{Question}
\newtheorem{definition}{Definition}
\newtheorem{theorem}{Theorem}
\newtheorem{lemma}{Lemma}
\newtheorem{corollary}{Corollary}

\title{Testing Markov Chains Without Hitting}
\author{%
 Yeshwanth Cherapanamjeri\\
 UC Berkeley\\
 \texttt{yeshwanth@berkeley.edu}
 \and
 Peter L. Bartlett \\
 UC Berkeley\\
 \texttt{peter@berkeley.edu}
}
\date{}

\begin{document}

\maketitle

\begin{abstract}
  We study the problem of identity testing of markov chains. In this setting, we are given access to a single trajectory from a markov chain with unknown transition matrix $\bm{Q}$ and the goal is to determine whether $\bm{Q} = \bm{P}$ for some known matrix $\bm{P}$ or $\distt(\bm{P}, \bm{Q}) \geq \epsilon$ where $\distt$ is suitably defined. In recent work by \cite{daskalakis2017testing}, it was shown that it is possible to distinguish between the two cases provided the length of the observed trajectory is at least super-linear in the hitting time of $\bm{P}$ which may be arbitrarily large.

  In this paper, we propose an algorithm that avoids this dependence on hitting time thus enabling efficient testing of markov chains even in cases where it is infeasible to observe every state in the chain. Our algorithm is based on combining classical ideas from approximation algorithms with techniques for the spectral analysis of markov chains.
\end{abstract}

\section{Introduction}

Statistical hypothesis testing is the principal method for lending statistical validity to claims made about the real world and is a vital step in any scientific enterprise. In the framework of statistical hypothesis testing, an investigator subjects hypotheses made as part of their inquiry by testing it against data collected from the real world. While the abstract framework of hypothesis testing is very powerful, its usefulness is limited by the range of hypotheses for which statistically efficient procedures have been developed. Furthermore, these tests also need to be computationally viable with large datasets. Unfortunately, most cases for which efficient procedures are known are concerned with the setting where we have access to independent and identically distributed observations from some underlying distribution. This severely restricts the use of these procedures.

Motivated by these considerations, recent work by \cite{daskalakis2017testing} studied the problem of identity testing of markov chains given a single trajectory where strong correlations may exist between successive samples. They propose an algorithm to test whether the transition matrix, $\bm{Q}$, underlying the observed trajectory is equal to a known transition matrix, $\bm{P}$, or sufficiently far from it. They propose a notion of difference between markov chains which takes into account the connectivity properties of the chain to ensure that the problem remains well posed. However, a major drawback of their approach is that their runtime depends on the hitting time of $\bm{P}$ and an open question from their work is whether this dependence is truly necessary and conjectured it was not.

The approach of \cite{daskalakis2017testing} is to convert the identity testing problem on markov chains to the simpler problem of identity testing of distributions given iid samples. The main idea is to use the observed trajectory to simulate samples from the distribution characterized by $\frac{1}{n} \bm{P}$. To simulate one sample from this distribution, one first picks a row of $\bm{P}$ uniformly at random and sample from the row using the trajectory. However, to generate the number of samples required to distinguish the two chains via this method, one needs to sample every row of $\bm{P}$ at least once with high probability. This leads to the dependence on the hitting time in the length of the observed trajectory.

In this work, we propose an algorithm for identity testing of markov chains that avoids the dependence on the hitting time of $\bm{P}$. That is, we would like to solve the identity testing problem even in settings where one may not even be able to observe all the states in the chain. Similar to \cite{daskalakis2017testing}, we reduce the identity testing problem on markov chains to simpler identity testing problems on distributions given iid samples. However, instead of a reduction to a single identity testing problem, we formulate several identity testing problems. Our main insight is that to distinguish two sufficiently different markov chains, it is sufficient to analyze the trajectory in subsets of states which are close to being disconnected from the rest of the state space but well connected within themselves. That is, we formulate for each such subset $S$, an identity testing problem whose solution also resolves the testing problem on markov chains. However, this approach is throttled by two main difficulties:

\begin{enumerate}
  \item Computing these ``high-information'' subsets and
  \item Ensuring we have sufficiently many samples from these subsets
\end{enumerate}

Our first main requirement of these subsets is that they have enough information to distinguish two different markov chains. We use as a sufficient criterion the property that these sets are poorly connected to the rest of the state space. The next crucial property that we will require is that the identity testing problem defined by the set can be simulated given a small number of samples from the set. We show that this property too can be related to the expansion properties of the set. This is guaranteed for a candidate set, $S$ if for all subsets, $R \subset S$, $R$ is well-connected to the rest of the set. Given these two requirements, our goal is to compute sets well connected within themselves and poorly connected to the rest of the state space. To do so, we generalize classical approximation algorithms for the Sparsest-Cut problem. However, this only ensures the first required property. To ensure the second property of being well connected within the set, we combine this approach with a divide and conquer framework. We then recursively extract such ``high-information'' subsets to obtain a partitioning of the state space into several ``high-information'' sets and a single ``low-information'' set.

To tackle the second problem of ensuring we have enough points from these ``high-information'' subsets in the observed trajectory, we use techniques from the spectral analysis of markov chains to show that the chain does not spend too much time in the ``low-information'' component of the chain. The failure of our graph partitioning algorithm to partition the ``low-information'' component means that all subsets of the ``low-information'' component are well connected to the rest of the state space. This ensures that the chain escapes from this component fairly quickly if it enters it. 

\textbf{Related Work: } In the statistics community, a variety of tests have been developed for distribution testing in the iid scenario: Cramer-von Mises (\cite{cramer1928composition}), $\chi^2$ (\cite{pearson1900x}), Kolmogorov-Smirnov (\cite{smirnov1939estimation}) and for more recent results, \cite{agresti2013categorical, d2017goodness}. However, the analysis of these methods pertains to the asymptotic distribution of the test statistic without finite sample guarantees. In the computer science community, there has been a flurry of recent work in this setting, with a focus on finite sample lower bounds and statistical and  computational tractability:  \cite{batu2004sublinear, acharya2015optimal, canonne2016testing, daskalakis2018distribution, daskalakis2017testing, valiant2011testing, chan2014optimal, diakonikolas2015testing, rubinfeld2009testing, valiant2011testing, valiant2013instance, valiant2017automatic, rubinfeld2012taming, blais2017distribution, batu2000testing, batu2001testing, paninski2008coincidence, acharya2015optimal, diakonikolas2016new, diakonikolas2016collision}.

The problem of identity testing and estimation in markov chains was, to the best of our knowledge, first studied in the seminal works of \cite{bartlett1951frequency, anderson1957statistical, billingsley1961statistical}. However, the results obtained are in the asymptotic regime with the number of samples tending to infinity. Recent work by \cite{daskalakis2017testing} provide finite sample analysis for the identity testing of markov chains but the length of the trajectory required depends on delicate connectivity properties of the chain like hitting times which may be arbitrarily large.

The sparsest cut problem has been intensely studied with the breakthrough result of \cite{leighton1999multicommodity} devising the first $O(\log n)$ approximation algorithm followed by a subsequent result by \cite{linial1995geometry} which interprets the algorithm from a metric embedding perspective (\cite{bourgain1985lipschitz}). The $O(\log n)$ barrier was subsequently improved to $O(\sqrt{\log n})$ in another beautiful result by \cite{arora2009expander}. These algorithms have been used in divide and conquer based approaches to several combinatorial problems (\cite{shmoys1997cut}) and constructing approximation algorithms for unique games (\cite{trevisan2005approximation}). While graph decomposition techniques have been studied previously (see, for example, \cite{spielman2004nearly, trevisan2005approximation, goldreich1999sublinear}), approaches based on spectral techniques yield weaker guarantees than those based on sparsest cut approximations. Graph decompositions based on \cite{leighton1999multicommodity} have been studied in \cite{trevisan2005approximation} however these results are not strong enough for our setting as they only imply the existence of internally well-connected partitions with potentially several ``low-information'' sets whereas, we crucially require that there exists at most one such set.

The relationship between the sparsest cut value of a markov chain and its spectral properties are well known (\cite{cheeger1969lower, sinclair1989approximate}) and have numerous applications (\cite{chung1997spectral,kannan1997random,lee2018kannan}). In the analysis of our algorithm, we use these techniques to bound hitting times in markov processes restricted to subsets of the state space and escape times from subsets of the state space where we bound the top eigenvalue of sub-matrices of the transition matrix as opposed to the second eigenvalue of the transition matrix.

\section{Preliminaries}

We denote scalar values by small letters such as $a$, vectors with bolded small letters such as $\bm{v}$ and matrices with bolded capital letters like $\bm{P}$. We use capital letters like $P,Q,R$ chiefly to denote subsets of $[n]$ and calligraphic capital letters $\mathcal{S}$ to denote sets of such subsets. For a vector $\bm{v}$, $v_i$ denotes the $i^{th}$ entry in the vector. For a matrix $\bm{P}$ and two subsets $R$ and $S$, $\bm{P}_i$ denotes the $i^{th}$ column of a matrix, $\bm{P}_{ij}$ denote the $j^{th}$ entry of the $i^{th}$ row of the matrix, $\bm{P}_{R,S}$  corresponds to the $\abs{R} \times \abs{S}$ sized sub-matrix corresponding to the rows in $R$ and columns in $S$ and $\bm{P}_R$ is used as shorthand for $\bm{P}_{R,R}$. We use $\ot$ and $\omt$ to hide logarithmic factors in $n$ and $\epsilon$. We use $\rho(\bm{M})$ to denote the largest eigenvalue of the matrix $\bm{M}$. We restate the definitions of the Total Variation and Hellinger distances (as stated in \cite{daskalakis2017testing}):

\begin{definition}
  \label{def:dheltv}
  For two distributions $\bm{p}$ and $\bm{q}$ over a support $[n]$, we have the Hellinger and Total Variation distances, denoted by $\dhel$ and $\dtv$ respectively, defined by:
  \begin{equation*}
    \dhel^2 (\bm{p}, \bm{q}) = \frac{1}{2} \sum_{i \in [n]} (\sqrt{\bm{p}_i} - \sqrt{\bm{q}_i})^2 = 1 - \sum_{i \in [n]} \sqrt{\bm{p}_i\bm{q}_i}, \qquad \dtv = \frac{1}{2} \sum_{i \in [n]} \abs{\bm{p}_i - \bm{q}_i}
  \end{equation*}
  Furthermore, the two distances enjoy the following relationship:
  \begin{equation*}
    \sqrt{2} \dhel(\bm{p}, \bm{q}) \geq \dtv (\bm{p}, \bm{q}) \geq \dhel^2 (\bm{p}, \bm{q})
  \end{equation*}
\end{definition}

Now, we will introduce some notations for markov chains:

\subsection{Markov Chains}
In this paper, we are only concerned with finite-dimensional markov chains:
\begin{definition}
  \label{def:fmc}
  A finite dimensional homogeneous markov chain is a stochastic process $\{X_t\}_{t \in \mb{N}}$ over a state space $[n]$ which satisfies the following property:
  \begin{equation*}
    \mb{P} \lbrb{X_{t + 1} = j | X_0 = i_0, \dots, X_{t - 1} = i_{t - 1}, X_t = i} = p_{i,j}
  \end{equation*}
  That is the probability of the state at time step $t + 1$ given the states from $X_{0}, \dots, X_t$ only depends on the previous time step and this transition probability does not depend on the specific time step $t$.
\end{definition}

We will use $\bm{w}$ to denote a finite sample from a markov chain and $\bm{w}_\infty$ to denote an infinite sample from the markov chain. We will denote the transition matrices of markov chains usually by $\bm{P}$ and $\bm{Q}$ and we will be concerned with the symmetric case where both matrices $\bm{P}$ and $\bm{Q}$ are symmetric. We will also assume that $\bm{P}$ and $\bm{Q}$ are irreducible. We will now, restate the distance measure between two transition matrices $\bm{P}$ and $\bm{Q}$ as stated in \cite{daskalakis2017testing}:

\begin{definition}[Distance between Markov Chains]
  For two symmetric transition matrices $\bm{P}$ and $\bm{Q}$, the distance between them is defined by:
  \begin{equation*}
    \dmc (\bm{P}, \bm{Q}) = 1 - \rho (\sq (\bm{P}, \bm{Q}))
  \end{equation*}
  where the function $\sq: \mb{R}_+^{n \times n} \times \mb{R}_+^{n \times n} \rightarrow \mb{R}_+^{n \times n}$ is defined by:
  \begin{equation*}
    (\sq (\bm{P}, \bm{Q}))_{ij} = \sqrt{\bm{P}_{ij} \bm{Q}_{ij}}
  \end{equation*}
\end{definition}

\begin{definition}
  \label{def:mcsub}
  Let $\bm{P}$ be a symmetric irreducible markov chain and $T \subset [n]$ be a subset of states. Now, let $\bm{Y} = Y_1, Y_2, \dots$ be a markov process with transition matrix $\bm{P}$ and let $\tau_1, \tau_2, \dots$ be defined such that:
  \begin{equation*}
    \tau_1 = \min \{j: Y_j \in T\},\qquad \tau_i = \min \{j: j > \tau_{i - 1} \wedge Y_j \in T\}
  \end{equation*}
  Then the sequence $\bm{X} = Y_{\tau_1}, Y_{\tau_2}, \dots $ is defined to be the markov process, $\bm{Y}$, observed on $T$.
\end{definition}

We will now state some definitions which we will relate to the spectral properties of the markov chain. The first definition is the notion of expansion of a set of states which intuitively measures how well the set of states is connected to the rest of the state space:

\begin{definition}[Expansion]
  Given a matrix, $\bm{P}$, with positive entries, the expansion of a set $S$, denoted by $h_{\bm{P}}(S)$ is defined as:
  \begin{equation*}
    h_{\bm{P}} (S) = \frac{\sum_{i \in S, j \notin S} \bm{P}_{ij}}{\min (\abs{S}, \abs{\bar{S}})}
  \end{equation*}
\end{definition}

The Cheeger constant of a markov chain is defined as the minimum of the expansion over all subsets of the state space.

\begin{definition}[Cheeger Constant]
  The Cheeger Constant of a Markov Chain with transition matrix, $\bm{P}$, is the minimum expansion of any subset of the state space.

  \begin{equation*}
    \chg(\bm{P}) = \min_{S \subset [n]} h_{\bm{P}} (S)
  \end{equation*}
\end{definition}

The following relationship between the Cheeger constant of a markov chain and the spectrum of its transition matrix is well known from the work of \cite{sinclair1989approximate}.

\begin{lemma}[\cite{sinclair1989approximate}]
  \label{lem:sinc}
  Let $\bm{P}$ be the transition matrix of a symmetric markov chain with eigen values $1 = \lambda_1 > \lambda_2 \geq \dots \geq \lambda_n \geq -1$. Furthermore, assume that $\bm{P}$ satisfies $\chg (\bm{P}) \geq \alpha > 0$. Then, we have:

  \begin{equation*}
    \lambda_2 \leq 1 - \frac{\alpha^2}{2}
  \end{equation*}
\end{lemma}

Now, we define the hitting time of a markov chain.

\begin{definition}
  \label{def:htt}
  Let $\bm{P}$ be the transition matrix of a markov chain, $\bm{X}$, over state space $[n]$. Let $\tau_j = \min\{t: X_t = j\}$. Then, the hitting time of $\bm{P}$, denoted by $\hitt (\bm{P})$ is defined as follows:
  \begin{equation*}
    \hitt (\bm{P}) = \max_{i, j \in [n]} \mb{E} [\tau_j | X_0 = i]
  \end{equation*}
\end{definition}

\subsection{Sparsest Cut}
Here, we will state some definitions relating to the graph decomposition algorithm we use for partitioning the state space of our markov chain. Our first definition is one that is closely related to the notion of expansion defined previously:

\begin{definition}[Cut Value]
  \label{def:cutVal}
  Given a non-negative matrix, $\bm{P}$, the Cut Value of a set $S$, is defined as:
  \begin{equation*}
    g_{\bm{P}} (S) = \frac{\sum_{i \in S, j \in \bar{S}} \bm{P}_{ij}}{\abs{S} \abs{\bar{S}}}
  \end{equation*}
\end{definition}

The Sparsest Cut problem is then defined as the problem of finding the set obtaining the minimum cut value over all subsets.

\begin{question}[Sparsest Cut]
  \label{que:spcut}
   Given a non-negative matrix $\bm{P}$, the goal is to find a subset $S^*$:

  \begin{equation*}
    S^* = \argmin_{S \subset [n]} g_{\bm{P}} (S)
  \end{equation*}
\end{question}

The Sparsest Cut problem is well known to be NP-Hard in general. However, good polynomial-time approximation algorithms are known to give a subset whose Cut Value which is within logarithmic factors of the Sparsest Cut value.

\section{Testing Markov Chains}
\label{sec:test}

In this section, we state and prove the main result of the paper. We introduce our algorithm for identity testing of markov chains and prove statistical and computational guarantees on its performance. As stated before, our algorithm follows the reduction framework of \cite{daskalakis2017testing} but instead of a reduction to a single distribution testing problem, we instead reduce the problem to multiple distinct distribution testing problems where each problem corresponds to a disjoint subset of the state space. The main insight of our algorithm is that to distinguish between two markov chains that are sufficiently far from each other, it is sufficient to perform a test in such ``high-information'' sets. Our algorithm proceeds along three main steps:

\begin{enumerate}
  \item \textbf{State Partitioning: } Partition the states into $S_1, \dots, S_k, T$ where the subsets $S_1, \dots, S_k$ are the ``high-information'' sets and $T$ is a single ``low-information'' set.
  \item \textbf{Generate IID Samples: } Check whether we have enough samples from one of the $S_i$ to generate samples for the iid distribution problem corresponding to it.
  \item \textbf{Run Identity Tester: } If so, return the result of the test or declare $\distt(\bm{P}, \bm{Q}) \geq \epsilon$.
\end{enumerate}

The full algorithm is described in Algorithm~\ref{alg:idtst} with supplementary algorithms for graph partitioning in Algorithms~\ref{alg:ec} and \ref{alg:pg} and to simulate iid samples in Algorithm~\ref{alg:geniid}. The main result of our paper is the following performance guarantee on Algorithm~\ref{alg:idtst}:

\begin{theorem}
  \label{thm:mainth}
  There is a polynomial time algorithm (Algorithm~\ref{alg:idtst}) which given access to $\ot\lprp{n / \epsilon^4}$ samples from a markov process with transition matrix $\bm{Q}$ and a symmetric transition matrix $\bm{P}$ correctly distinguishes between the two cases:
  \begin{equation*}
    \text{Case 1: } \bm{Q} = \bm{P}, \qquad \text{Case 2: } \distt(\bm{Q}, \bm{P}) \geq \epsilon
  \end{equation*}
  with probability at least $2/3$.
\end{theorem}

We start by giving a description of the type of the distributions for which we will employ our iid distribution tester:

\begin{definition}
  \label{def:ddef}
  Let $\bm{P}$ be the transition matrix of a symmetric markov chain and let $R$ be a subset of its states, then we have the distribution $\distt(R, \bm{P})$ defined over a support of size $\abs{R}^2 + 1$ composed of $\{(i,j): i,j \in R\} \cup \{\nue\}$ where $\nue$ denotes a element which is none of the elements $(i,j),\ i,j \in [n]$:
  \begin{equation*}
    \forall i,j \in R, (\distt (R, \bm{P})) ((i,j)) = \frac{\bm{P}_{ij}}{\abs{R}}, \qquad (\distt (R, \bm{P})) (\nue) = 1 - \frac{1}{\abs{R}} \sum_{i,j \in R} \bm{P}_{ij}
  \end{equation*}
\end{definition}

Therefore, given a partitioning of the state space such that distributions of the above type are sufficiently different, it suffices to have enough samples from any one of the partitions. However, there are two questions that need to be answered before we can proceed:

\begin{enumerate}
  \item Which partitions of the state space should we use to define such distributions?
  \item How many samples does one need from each of the partitions?
\end{enumerate}

It turns out that answers to both questions depend on the expansion properties of the sets. However, for the first property, we would like to have sets of low expansion, that is, subsets of the state space that are poorly connected to the rest of the state space while for the second property, we would like sets which are well connected internally. We will see that the second property relates to the hitting time of the markov process defined on the specific subset of states which is small if the original subset is well connected within itself. Therefore, one would like decompositions of the state space which are poorly connected to the rest of the state space but are well connected within themselves.

We would like to point out that conventional graph decomposition algorithms decompose the graph into subsets which are well connected internally while removing a very small number of edges which guarantees the first property for a large fraction of the subsets in terms of total number of states. However, for the remaining subsets, we have no such guarantees and therefore, it is unclear whether samples from such subsets can be used to distinguish the two markov chains. Even though one can guarantee that upon entering such subsets, the trajectory is likely to quickly leave the subset, one cannot guarantee that the next partition that the chain visits is a ``high-information'' subset. An alternate approach is to group all such ``low-information'' subsets into a single set but in this case, one loses the expansion guarantees of the individual sets which again makes it hard to bound the amount of time needed to escape from this set.

In light of all the above mentioned difficulties, we devise a new graph partitioning algorithm which decomposes the graph into potentially several well connected ``high-information'' sets and a single well connected ``low-information'' set from which one can guarantee that we escape from quickly and therefore reach a ``high-information'' set. We generalize conventional linear programming relaxations for the sparsest cut problem to respect component constraints and then use the above generalization to recursively partition the graph into subsets while measuring sparsest cut values with respect to the original graph instead of sub-graphs formed after removing partitions. The full details of our algorithm are deferred to the Appendix (See Algorithms~\ref{alg:ec} and \ref{alg:pg}).

Note that following the approach of \cite{daskalakis2017testing}, we can sample from $\distt(T, \bm{P})$ given access to an infinite word. Firstly, note that it is possible to sample from $\distt(T, \bm{P})$ by first sampling an element from $T$ and then sampling from the distribution corresponding to the sampled element in $\distt(T, \bm{P})$. Therefore, to obtain $l$ samples from $\distt(T, \bm{P})$, we start by first generating $l$ samples from $\text{Uniform}(T)$. Let the number of times we generated state $i\in T$ be denoted by $r_i$. Now, we simply scan the infinite word sequentially and each time we encounter an element $j \in T$ at position $t$, we reject the sample if $j$ has been encountered more than $r_j$ times or add the transition $j \rightarrow w_{t + 1}$ to our samples if $w_{t + 1} \in T$ or add $\nue$ to our samples if $w_{t + 1} \notin T$. The correctness of the described procedure follows from the markov property which ensures that all the transitions generated previously are independent of the ones generated after. The procedure is formally described in Algorithm~\ref{alg:geniid}.

\begin{algorithm}[H]
  \caption{Generate IID Samples}
  \label{alg:geniid}
  \begin{algorithmic}[1]
    \STATE \textbf{Input}: Finite word $\bm{w} \in [n]^m$, Subset $T$, Number of samples $l$

    \STATE $\bm{v} \leftarrow l\text{ samples from Uniform}(T)$
    \STATE $\bm{r} \leftarrow \text{Histogram}(\bm{v})$
    \STATE $\mathcal{S} \leftarrow \{\}$
    \FOR {$i = 1:m-1$}
      \STATE $j \leftarrow w_i$
      \IF {$\bm{r}_j > 0$ and $w_{i + 1} \in T$}
        \STATE $\mathcal{S} \leftarrow \mathcal{S} \cup (j, w_{i + 1})$
      \ELSIF {$\bm{r}_j > 0$ and $w_{i + 1} \notin T$}
        \STATE $\mathcal{S} \leftarrow \mathcal{S} \cup \nue$
      \ENDIF
      \STATE $\bm{r}_j \leftarrow \bm{r_j} - 1$
    \ENDFOR

    \IF {$\exists i \in T: \bm{r}_i > 0$}
      \STATE \textbf{Return: } False
    \ENDIF

    \STATE \textbf{Return: } $\mathcal{S}$
  \end{algorithmic}
\end{algorithm}

\begin{algorithm}[H]
  \caption{Identity Test of Markov Chains}
  \label{alg:idtst}
  \begin{algorithmic}[1]
    \STATE \textbf{Input}: Finite word $\bm{w} \in [n]^m$, Target Transition Matrix $\bm{P}$, Target Accuracy $\epsilon$

    \STATE $(\mathcal{S}, T) \leftarrow \pg (\bm{P}, \epsilon / 16)$
    \FOR {$S \in \mathcal{S}$}
      \STATE $l^\prime \leftarrow O\lprp{\frac{\abs{S} \log (n)}{\epsilon^2}}$
      \STATE $\mathcal{R}_S \leftarrow \giid (\bm{w}, S, l^\prime)$
      \IF {$\mathcal{R}_S \neq \text{False}$}
        \STATE \textbf{Return: } $\itest (\mathcal{R}_S, \distt(S, \bm{P}), \epsilon^2 / 32)$
      \ENDIF
    \ENDFOR

    \STATE \textbf{Return: } False
  \end{algorithmic}
\end{algorithm}

\section{Proof}
\label{sec:mpr}

In this section, we will present the proof of Theorem~\ref{thm:mainth}. As mentioned before, the guarantees provided by conventional graph partitioning algorithms are not strong enough to ensure the small trajectory lengths required for the success of Theorem~\ref{thm:mainth}. In the first subsection, we will describe some key lemmas relating to the graph decomposition technique detailed in Algorithm~\ref{alg:pg}.

\subsection{Markov Chain Decomposition}

This first lemma, proved in Appendix~\ref{prf:pg}, describes the expansion properties of the partition of the markov chain state space obtained from Algorithm~\ref{alg:pg}. Intuitively, it decomposes the graph into a set of subsets $\mathcal{S}$ which consists of sets which are well connected within themselves but poorly connected to the rest of the state space and a single set $T$ in which every subset is well connected to the rest of the state space. The sets in $\mathcal{S}$ refer to the ``high-information'' sets alluded to previously while the set $T$ is the single ``low-information'' subset of the state space.

\begin{lemma}
  \label{lem:pg}
  Algorithm~\ref{alg:pg} returns a tuple $(\mathcal{S}, T)$ such that we have for all $S \in \mathcal{S}$:
  \begin{equation*}
    \text{Claim 1: } \frac{\sum_{i, j \in S} \bm{P}_{ij}}{\abs{S}} \geq 1 - \beta
    \qquad \text{Claim 2: } \forall R \subset \S\ \frac{\sum_{i \in R, j \in (\S \setminus R)} \bm{P}_{ij}}{\min(\abs{R}, \abs{\S \setminus R})} \geq \Omega \lprp{\frac{\beta}{\log^2 n}}
  \end{equation*}
  And $T$ satisfies:

  \begin{equation*}
    \text{Claim 3: } \forall R \subseteq T\ \frac{\sum_{i \in R, j \in \bar{R}} \bm{P}_{ij}}{\abs{R}} \geq \Omega \lprp{\frac{\beta}{\log n}}
  \end{equation*}

  Furthermore, the subsets in $\mathcal{S}$ along with $T$ form a partition of $[n]$.
\end{lemma}

Our next lemma, proved in Appendix~\ref{prf:helbnd}, shows that the distributions from Definition~\ref{def:ddef} are far in Hellinger distance if the original markov chains are far.

\begin{lemma}
  \label{lem:helbnd}
  Let $\bm{P}$ and $\bm{Q}$ be transition matrices of symmetric markov chains such that $\dmc(\bm{P}, \bm{Q}) \geq \epsilon$. Suppose now, that $T \subseteq [n]$ satisfies:

  \begin{equation*}
    \frac{\sum_{i, j \in T} \bm{P}_{ij}}{\abs{T}} \geq 1 - \frac{\epsilon}{16}
  \end{equation*}

  Then, we have:

  \begin{equation*}
    \dhel^2 \lprp{\distt(T, \bm{P}), \distt(T, \bm{Q})} \geq \frac{\epsilon^2}{32}
  \end{equation*}
\end{lemma}

In the next lemma, whose proof may be found in Appendix~\ref{prf:submc}, we analyze the spectral properties of the markov processes observed on a subset of states. This lemma will be crucial in bounding the number of samples we need to see from this subset in order to generate a large number of samples from the distribution corresponding to this subset.

\begin{lemma}
  \label{lem:submc}
  Let $\bm{P}$ be a symmetric irreducible markov chain and $T \subset [n]$ be a subset of states. Let $\bm{Y} = Y_1, Y_2, \dots$ be a markov process with transition matrix $\bm{P}$ and let $\bm{X} = X_1, X_2, \dots$ be the markov process observed on the subset $T$. Then, $\bm{X}$ is also a symmetric markov process with transition matrix:
  \begin{equation*}
    \bm{Q} = \bm{P}_T + \sum_{i = 1}^\infty \bm{P}_{T, \bar{T}} \bm{P}_{\bar{T}}^i \bm{P}_{\bar{T}, T}
  \end{equation*}
\end{lemma}

The next corollary is an application of Lemma~\ref{lem:submc} to markov processes defined on the ``high-information'' sets by exploiting their good expansion properties within the set itself. Its proof may be found in Appendix~\ref{prf:expsmc}.
\begin{corollary}
  \label{cor:expsmc}
  In the setting of Lemma~\ref{lem:submc}, suppose in addition that $T$ satisfies:
  \begin{equation*}
    \forall R \subset T\ \frac{\sum_{i \in R, j \in (T \setminus R)} \bm{P}_{ij}}{\min(\abs{R}, \abs{T \setminus R})} \geq \alpha
  \end{equation*}
  Then, the transition matrix $\bm{Q}$ of the chain $\bm{X}$ satisfies:
  \begin{equation*}
    \chg(\bm{Q}) \geq \alpha
  \end{equation*}
\end{corollary}

We now bound the hitting time of markov processes defined on ``high-information'' subsets. See Appendix~\ref{prf:htb} for the proof.
\begin{lemma}
  \label{lem:htb}
  Let $\bm{P}$ be the transition matrix of a symmetric markov chain, over state space $[n]$, satisfying $\chg(\bm{P}) \geq \alpha > 0$. Then, the hitting time of $\bm{P}$ is bounded as follows:

  \begin{equation*}
    \text{HitT}(\bm{P}) \leq \ot \lprp{\frac{n}{\alpha^2}}
  \end{equation*}
\end{lemma}

The next lemma, which is a consequence of Theorem~1 from \cite{daskalakis2018distribution} (Also stated in \cite{daskalakis2017testing}), bounds the number of samples required to distinguish two distributions over the same support given a lower bound on their Hellinger distance.

\begin{lemma}
  \label{lem:idtst}
  Given a discrete distribution $p$ on $[n]$ and given access to i.i.d samples from a distribution $q$ with the same support, there is a tester which can distinguish whether $p = q$ or $\dhel(p, q) \geq \epsilon$ with $O(\frac{\sqrt{n}}{\epsilon^2} \log 1 / \delta)$ samples and failure probability at most $\delta$.
\end{lemma}

In the next lemma, we show how the expansion properties of the ``low-information'' set obtained before can be used to obtain a guarantee on the number of samples observed from the ``High-information'' sets. To prove the below bound, we bound the spectral norm of $\bm{P}_T$ which controls the amount of time needed to escape from the set $T$. Our proof mirrors that of Lemma~3.3 in \cite{sinclair1989approximate} but we bound the first eigenvalue of a sub-matrix as opposed to the second eigenvalue of the whole transition matrix. The full details of the proof are deferred to Appendix~\ref{prf:tli}.

\begin{lemma}
  \label{lem:tli}
  Let $\bm{P}$ be the transition matrix of a symmetric markov chain. Furthermore, let $T\subset [n]$ be such that:
  \begin{equation*}
    \forall R \subseteq T,\ \frac{\sum_{i \in R, j \in \bar{R}} \bm{P}_{ij}}{\abs{R}} \geq \alpha
  \end{equation*}
  Then, in a word of length $l \geq 16\frac{1}{\alpha^2} \log (n) \log (1 / \delta)$, we have:
  \begin{equation*}
    \sum_{i = 1}^l \bm{1}\{X_i \notin T\} \geq \frac{l}{8 \log n \alpha^2}
  \end{equation*}
  with probability at least $1 - \delta$.
\end{lemma}

The next lemma from \cite{daskalakis2017testing} lower bounds the number of times we observe a certain state in a suitably long trajectory of a markov chain. We will use the lemma below for sub-chains consisting of chains corresponding to the ``high-information'' sets.
\begin{lemma}
  \label{lem:mxtb}
  Let $X_1, \dots, X_m$ be a word of length $m$ from an irreducible markov chain, over state space $[n]$ and transition matrix $\bm{P}$. Then for $m \geq \ot(\hitt (\bm{P})\log \hitt (\bm{p}))$, we have:
  \begin{equation*}
    \bm{P} \lbrb{\exists i: \abs*{\{t: X_t = i\}} \leq \frac{m}{8en}} \leq \frac{\epsilon^2}{n}
  \end{equation*}
  where the probability is over the sampling of $X_1, \dots, X_m$.
\end{lemma}

\subsection{Sample Generation Phase}

In this subsection, we will state and prove key lemmas relating to the sample generation phase of the algorithm. Here, we will assume that the observed word $\bm{w}$ is a subset of an infinite word $\bm{w}_\infty$ from a markov process with the same starting distribution and transition matrix. We will first analyze the sample generation process on the infinite word $\bm{w}_\infty$. Assuming that we have access to the infinite word $\bm{w}_\infty$, we see that the sample generation process will never fail as we see each state infinitely many times with probability $1$. In the first lemma, we show that given access to $\bm{w}_\infty$, we will be able to use any of the ``high-information'' sets to test between the two chains:

\begin{lemma}
  \label{lem:infFail}
  Suppose $(\mathcal{S}, T)$ is a decomposition of a markov chain $\bm{P}$ obtained from Algorithm~\ref{alg:idtst} and that we are given an infinite word $\bm{w}_\infty$ from a markov process with transition matrix $\bm{Q}$ and we are guaranteed one of the following two cases:
  \begin{equation*}
    \text{Case 1: } \distt (\bm{P}, \bm{Q}) \geq \epsilon \qquad \text{Case 2: } \bm{P} = \bm{Q}
  \end{equation*}
  Now, for each set $S \in \mathcal{S}$, let $l_S = \omt (\abs{S} / \epsilon^2)$, let $\mathcal{R}_S = \giid (\bm{w}_\infty, S, l_S)$. Then, we have:

  \begin{equation*}
    \mb{P} \lbrb{\exists S \in \mathcal{S}: \itest (\mathcal{R}_S, \distt (S, \bm{P}), \epsilon^2 / 32) \neq \bm{1} \lbrb{\bm{P} = \bm{Q}}} \leq \frac{1}{10}
  \end{equation*}
\end{lemma}

\begin{proof}
  We will first consider a single set $S \in \mathcal{S}$. In the case that $\bm{P} = \bm{Q}$, we have that $\mathcal{R}_S$ consists of $l_S$ samples from $\distt(S, \bm{P})$. Therefore, we have from the guarantees of $\itest$ from Theorem~\ref{lem:idtst} that

  \begin{equation*}
    \mb{P} \lbrb{\itest(\mathcal{R}_S, \distt(S, \bm{P})) = 1} \geq 1 - \frac{1}{10n}
  \end{equation*}

  In the alternate case where $\distt(\bm{P}, \bm{Q}) \geq \epsilon$, we have from Lemma~\ref{lem:helbnd} that $\dhel^2 (\distt(S, \bm{P}), \distt(S, \bm{Q})) \geq \epsilon^2 / 32$. Therefore, we have again from Lemma~\ref{lem:idtst}:

  \begin{equation*}
    \mb{P} \lbrb{\itest(\mathcal{R}_S, \distt(S, \bm{P})) = 0} \geq 1 - \frac{1}{10n}
  \end{equation*}

  The above two inequalities imply that for a fixed $S \in \mathcal{S}$, we have:
  \begin{equation*}
    \mb{P} \lbrb{\itest(\mathcal{R}_S, \distt(S, \bm{P})) = \bm{1} \lbrb{\bm{P} = \bm{Q}}} \geq 1 - \frac{1}{10n}
  \end{equation*}

  We note that since each $S \in \mathcal{S}$ is non-empty and along with $T$, they form a partition of $[n]$, there are at most $n$ sets in $\mathcal{S}$. Taking an union bound over the at most $n$ sets in $\mathcal{S}$, we get:

  \begin{equation*}
    \mb{P} \lbrb{\exists S \in \mathcal{S}: \itest (\mathcal{R}_S, \distt (S, \bm{P}), \epsilon^2 / 32) \neq \bm{1} \lbrb{\bm{P} = \bm{Q}}} \leq \frac{1}{10}
  \end{equation*}
\end{proof}

The above lemma shows that if we are able to generate samples from even one of the subsets $S \in \mathcal{S}$, we will be able to correctly answer the identity testing problem with high confidence. Therefore, to ensure the correctness of Algorithm~\ref{alg:idtst}, we simply need to show that the probability of being able to generate enough samples from the distribution corresponding to at least one of the sets $S \in \mathcal{S}$ is large. The next lemma, proved in the Appendix~\ref{prf:mxk}, is used to bound the number of times we will sample a particular state in the running of Algorithm~\ref{alg:geniid}.

\begin{lemma}
  \label{lem:mxk}
  Let $X_1, \dots, X_m$ be $m$ iid samples from $\text{Uniform}([k])$. Let $\bm{v} = \text{Histogram} (X_1, \dots, X_m)$. Suppose further that $m \geq 10 k \log (n / \epsilon)$ for some $n > k$. Then, we have:
  \begin{equation*}
    \max_{i \in [k]} v_i \leq 2 \frac{m}{k}
  \end{equation*}
  with probability at least $1 - \frac{\epsilon}{n^2}$.
\end{lemma}

In the following lemma, we show that the number of samples in a trajectory from $S \in \mathcal{S}$ we will need to observe to generate $l_S$ samples from $\distt(S, \bm{P})$ is small.

\begin{lemma}
  \label{lem:subMcSamps}
  Suppose $(\mathcal{S}, T)$ is a decomposition of a markov chain $\bm{P}$ obtained in Algorithm~\ref{alg:idtst} and that $\bm{w}_\infty$ is an infinite length trajectory from a markov process with transition matrix $\bm{P}$. Now for each $S \in \mathcal{S}$, let $l_S = \ot(\abs{S} / \epsilon^2)$ and let $w_{\tau^S_1}, w_{\tau^S_2}, \dots , w_{\tau^S_{N_S}}$ be the indices corresponding to the entries in $S$ encountered in the running of $\giid(\bm{w}_\infty, S, l_S)$. Then we have:

  \begin{equation*}
    \mb{P} \lbrb{\forall S \in \mathcal{S}: N_S \leq \ot (\abs{S} / \epsilon^2)} \geq \frac{9}{10}
  \end{equation*}
\end{lemma}
\begin{proof}
  As in the proof of Lemma~\ref{lem:infFail}, we first consider a single component $S \in \mathcal{S}$. Note that the trajectory $\bm{w}_\infty$ observed on the set of states in $S$, $\bm{w}^S_\infty$, is also a markov process. Furthermore, we know from Lemmas~\ref{lem:pg}, \ref{lem:htb} and Corollary~\ref{cor:expsmc} that the hitting time of $\bm{w}^S_\infty$ is $\ot (\abs{S} / \epsilon^2)$. Therefore, we have from Lemma~\ref{lem:mxtb}, that in a trajectory of length $N_S$ from $\bm{w}^S_\infty$, we have:

  \begin{equation*}
    \mb{P} \lbrb{\exists i: \abs*{\{t: X_t = i\}} \leq \frac{N_S}{8e\abs{S}}} \leq \frac{1}{20 n}
  \end{equation*}

  Similarly, to generate $l_S$ samples from $\distt(S, \bm{P})$, the maximum number of times any a particular state in $S$ will be sampled in a run of Algorithm~\ref{alg:geniid}, denoted by $m_S$, is upper bounded by Lemma~\ref{lem:mxk}:

  \begin{equation*}
    \mb{P} \lbrb{m_S \geq 2 \frac{l_S}{\abs{S}}} \leq \frac{1}{20n}
  \end{equation*}

  Therefore, the probability that we succeed in generating $l_S$ samples from $\distt(S, \bm{P})$ is upper bounded by the probability that both the above events fail to occur as this implies the event ${\{\forall i: \abs{\{t: X_t = i\}} \geq m_S\}}$ ensuring the sample generation process succeeds. Therefore, we have:
  \begin{equation*}
    \mb{P} \lbrb{N_S \geq \ot (\abs{S} / \epsilon^2)} \leq \frac{1}{10n}
  \end{equation*}

  By taking a union bound over the at most $n$ subsets $S \in \mathcal{S}$, we get:
  \begin{equation*}
    \mb{P} \lbrb{\forall S \in \mathcal{S}: N_S \leq \ot (\abs{S} / \epsilon^2)} \geq \frac{9}{10}
  \end{equation*}
\end{proof}

\subsection{Proof of Theorem~\ref{thm:mainth}}

We are now ready to prove Theorem~\ref{thm:mainth}. We will prove the theorem in two cases: 

\textbf{Case 1: } $\bm{P} = \bm{Q}$. In this case, we see that the Algorithm~\ref{alg:idtst} only outputs the wrong answer if the sample generation process in Algorithm~\ref{alg:geniid} fails for all subsets $S \in \mathcal{S}$ or if the sample generation process succeeds but \itest{} returns the wrong answer. We will first upper bound the probability that the sample generation process fails. To do this, we see from Lemma~\ref{lem:tli} that if we have a trajectory of length $m \geq \omt(m / \epsilon^4)$, then we have:

\begin{equation*}
  \sum_{i = 1}^m \bm{1} \lbrb{X_i \notin T} \geq \omt \lprp{\frac{n}{\epsilon^2}}
\end{equation*}

with probability at least $0.9$. Therefore, we have with probability at least $0.9$, there exists at least one set $S \in \mathcal{S}$:

\begin{equation*}
  \sum_{i = 1}^m \bm{1} \lbrb{X_i \in S} \geq \omt \lprp{\frac{\abs{S}}{\epsilon^2}}
\end{equation*}

Therefore, the probability that the sample generation process fails is at most:

\begin{equation*}
  \mb{P} \lbrb{\exists S \in \mathcal{S}: N_S \geq \omt\lprp{\frac{\abs{S}}{\epsilon^2}}} + \mb{P} \lbrb{\forall S \in \mathcal{S}: \sum_{i = 1}^m \bm{1} \lbrb{X_i \in S} \leq \omt \lprp{\frac{\abs{S}}{\epsilon^2}}} \leq \frac{2}{10}
\end{equation*}

where $N_S$ and the bound on the first term are from Lemma~\ref{lem:subMcSamps}. We finally bound the probability of failure of the algorithm by the sum of the probabilities of the sample generation process failing and the probability of the \itest{} failing on samples generated from the infinite word, $\bm{w}_\infty$ from Lemma~\ref{lem:infFail}. Putting the two bounds together, we get that Algorithm~\ref{alg:idtst} fails with probability at most $0.7$ which is less than $2 / 3$.

\textbf{Case 2: } $\distt(\bm{P}, \bm{Q}) \geq \epsilon$. In this case, we see that the Algorithm~\ref{alg:idtst} always returns the correct answer if the sample generation process fails. Therefore, the probability of failure is at most the probability that \itest{} failing on samples generated from the infinite word from $\bm{Q}$. From Lemma~\ref{lem:infFail}, we know that this is at most $0.1$. Therefore, the failure probability of the Algorithm in this case is at most $0.9$ which is less than $2/3$.

The above two cases conclude the proof of the theorem.

\qed
\section{Conclusion}
\label{sec:conc}

We have presented an algorithm for identity testing of markov chains which avoids any dependence on brittle connectivity properties like the hitting time resolving a open question from \cite{daskalakis2017testing}. However, there are several open questions potentially relating to identity testing and graph partitioning arising from this work:

\begin{enumerate}
    \item The sample complexity of our approach $\ot(n / \epsilon^4)$ is sub-optimal in its dependence on the error parameter $\epsilon$. Can our approach be improved to the $\Omega (n / \epsilon)$ lower bound for the problem established in \cite{daskalakis2017testing}. 
    \item One reason for this dependence on $\epsilon$ is due to the graph partitioning algorithm which guarantees sets of low expansion. Is it possible to improve upon such graph partitioning algorithms or devise new graph partitioning algorithms to achieve improved error dependence?
    \item Markov chains are arguably the simplest possible model for sequential data analysis. How can we quantify distances between models for more complicated methods? What assumptions does one need to place on the model to ensure that statistical and computational efficiency is possible for such hypothesis testing tasks?
\end{enumerate}
\bibliography{refs}
\bibliographystyle{alpha}
\newpage
\appendix
\section{Decomposing a Markov Chain into Well Connected Components}
\label{sec:mcd}
\subsection{Sparsest Cut with Component Constraints}

In this section, we will design and analyze an algorithm for decomposing the state space of a Markov Chain into components that are internally well connected but poorly connected to the rest of the state space. Our algorithm is based on generalizations to the classical linear programming relaxations of the Sparsest Cut problem which is known to be NP-Hard in general. We will start by stating some classical results used to analyze such relaxations and adapt them to our setting. Our first result is Bourgain's famous metric-embedding theorem:

\begin{theorem}[\cite{bourgain1985lipschitz,linial1995geometry}]
  \label{thm:bme}
  Let $\mathcal{X}$ be a finite metric space of size $n$ endowed with a metric $d$. Then, there exists a function $f: \mathcal{X} \rightarrow \mb{R}^m$ and a constant $C > 0$ such that:

  \begin{equation*}
    \forall x,y \in \mathcal{X},\ d(x,y) \leq \norm{f(x) - f(y)}_1 \leq C \log n\, d(x,y)
  \end{equation*}

  And furthermore, $m$ is at most $O(\log^2 n)$ and can be found in randomized polynomial time.
\end{theorem}

We will now describe the linear programming relaxation to the Sparsest Cut problem. Before we describe the formulation, we first introduce the notion of a Cut Metric:

\begin{definition}[Cut Metric]
  \label{def:cutmet}
  For a state space $[n]$, the Cut Metric associated with a subset $S \subset [n]$ is defined as follows:

  \begin{equation*}
    \delta_S (i, j) = \begin{cases}
                        0, &\text{if $i,j \in S$ or $i,j \in \bar{S}$} \\
                        1, &\text{otherwise}
                      \end{cases}
  \end{equation*}
\end{definition}

It follows that the Cut Value of a subset can be restated in terms of the cut metric corresponding to the subset as follows:

\begin{equation*}
  g_{\bm{P}} (S) = \frac{\sum_{i,j \in [n]} \bm{P}_{ij} \delta_S (i,j)}{\sum_{i,j \in [n]} \delta_S (i,j)}
\end{equation*}

The Linear Programming relaxation to the Sparsest Cut problem, can now be seen naturally as broadening the class of metrics in the Sparsest Cut formulation from the set of cut metrics to the set of all metrics and is described below:

\begin{gather*}
  \min \sum_{i, j \in [n]} \delta_{ij} \bm{P}_{ij} \\
  \text{such that } \delta_{ii} = 0\, \forall i\\
  \delta_{ij} \leq \delta_{ik} + \delta_{kj} \, \forall i,j,k \\
  \sum_{i,j \in [n]} \delta_{ij} = 1 \\
  \delta_{ij} \geq 0 \label{eq:lpcut} \tag{\textbf{LP-CUT}}
\end{gather*}

where the second constraint is a normalization factor.

We will work with a natural variant of the sparsest cut problem where we are given a priori a subset $T$ of states all of which we require to be in the same component:

\begin{question}{Sparsest Cut with Component Constraints (\spccc):}
  \label{que:spccc}
  Given a non-negative matrix, $\bm{P}$ and a set of states $T$ that are all required to be in the same component, we define the Sparsest Cut Problem with Component Constraints as follows:

  \begin{equation*}
    S^* = \argmin_{T \subseteq S \subset [n]} \frac{\sum_{i, j \in [n]} \delta_S (i,j) \bm{P}_{ij}}{\sum_{i, j \in [n]} \delta_S (i,j)}
  \end{equation*}
\end{question}

Now, we give our Linear Programming relaxation of the \spccc{} problem. As for the Sparsest Cut problem, we relax the class of metrics beyond Cut Metrics, but we include the constraint that the distance between vertices in $T$ is $0$ and all the vertices in $T$ have the same distance to every other vertex:

\begin{gather*}
  \min \sum_{i, j \in [n]} \delta_{ij} \bm{P}_{ij} \\
  \text{such that } \delta_{ii} = 0\, \forall i\\
  \delta_{ij} \leq \delta_{ik} + \delta_{kj} \, \forall i,j,k \\
  \sum_{i,j \in [n]} \delta_{ij} = 1 \\
  \delta_{ij} \geq 0 \\
  \delta_{ij} = 0 \, \text{if $i,j \in T$} \\
  \delta_{ik} = \delta_{jk} \, \forall i,j \in T, k \in [n]\label{eq:lpccc} \tag{\textbf{LP-CCC}}
\end{gather*}

The last two constraints in the relaxation defined above ensure that there is no distance between any two states in $T$ and the distance from the states in $T$ to every other state is the same. We will now denote by $(\delta, v) = \lpccc(\bm{P}, T)$ a pair of metric $\delta$ and a value $v$ returned by \lpccc. We will now prove a lemma showing that the function $f$ guaranteed by Theorem~\ref{thm:bme} can be shown to have special structure.

\begin{lemma}
  \label{lem:me}
  Given an instance of the \spccc{} problem, $(\bm{P}, T)$ and solution $(\delta, v) = \lpccc(\bm{P}, T)$, there exists a function $f: V \rightarrow \mb{R}^m$ and a constant $C > 0$ such that:

  \begin{equation*}
    \text{Claim 1}: \delta_{ij} \leq \norm{f(i) - f(j)}_1 \leq C \log n\, \delta_{ij}, \qquad \text{Claim 2}: f(i) = f(j)\ \forall i,j \in T
  \end{equation*}

  Furthermore, $m$ is at most $O(\log^2 n)$ and $f$ can be found in randomized polynomial time.
\end{lemma}

\begin{proof}
  Let $f$ be the function whose existence is guaranteed by Theorem~\ref{thm:bme}. Note that $f$ satisfies Claim 1 of the lemma. For Claim 2, let $i,j \in T$. We know from the constraints on \lpccc{} that $\delta_{ij} = 0$. Therefore, from Theorem~\ref{thm:bme}, we may again conclude that:
  \begin{equation*}
    \norm{f(i) - f(j)}_1 = 0 \implies f(i) = f(j)
  \end{equation*}
  Thus proving Claim 2.
\end{proof}

The next lemma from \cite{linial1995geometry} shows that it is possible to express the $l_1$ metric defined by $f$ on the state space as a sum of cut metrics.

\begin{lemma}[\cite{linial1995geometry}]
  \label{lem:l1cut}
  Given $f: [n] \rightarrow \mb{R}^m$, it is possible to find in time $poly(n,m)$ a polynomial number of subsets $S_1, \dots, S_r$ and associated constants $\alpha_{S_i} > 0$ such that:

  \begin{equation*}
    \norm{f(j) - f(k)}_1 = \sum_{i = 1}^r \alpha_{S_i} \delta_{S_i} (j, k) \ \forall j,k \in [n]
  \end{equation*}
\end{lemma}

Now, finally, we conclude that the integrality gap of the Linear Programming Relaxation \lpccc{} is small and furthermore, a cut obtaining such a value can be found efficiently.

\begin{theorem}
  \label{thm:spccc}
  Given an instance of the \spccc{} problem $(\bm{P}, T)$, there exists a polynomial time algorithm, \fc{} which returns a cut $S^*$ satisfying:

  \begin{equation*}
    \frac{\sum_{i, j \in [n]} \delta_{S^*} (i,j) \bm{P}_{ij}}{\sum_{i, j \in [n]} \delta_{S^*} (i,j)} \leq O(\log n) \min_{T \subseteq S \subset V} \frac{\sum_{i, j \in [n]} \delta_{S} (i,j) \bm{P}_{ij}}{\sum_{i, j \in [n]} \delta_{S} (i,j)}
  \end{equation*}

  Furthermore, we have that $T \cap S^* = \phi$
\end{theorem}

\begin{proof}
  First, let $(\delta, v) = \lpccc(\bm{P}, T)$ and let $f$ be the function whose existence is guaranteed by Lemma~\ref{lem:me}. Furthermore $S_1, \dots, S_r$ denote the cuts with the associated constants $\alpha_{S_r} > 0$ as obtained from Lemma~\ref{lem:l1cut}. Now, we have:

  \begin{align*}
    \min_{i \in [r]} \frac{\sum_{j,k \in [n]} \delta_{S_i} (j, k) \bm{P}_{jk}}{\sum_{j,k \in [n]}\delta_{S_i} (j,k)} &\leq \frac{\sum_{i = 1}^r \alpha_{S_i}\sum_{j,k \in [n]} \delta_{S_i} (j, k) \bm{P}_{jk}}{\sum_{i = 1}^r \alpha_{S_i}\sum_{j,k \in [n]}\delta_{S_i} (j,k)} \\
    &= \frac{\sum_{j,k \in [n]} \norm{f(j) - f(k)}_1 \bm{P}_{jk}}{\sum_{j, k \in [n]} \norm{f(j) - f(k)}_1} \leq O(\log n) v
  \end{align*}

  where the first inequality follows from the fact that $\min_i \{\frac{a_i}{b_i}\} \leq \frac{\sum a_i}{\sum b_i}$ and the final inequality follows by applying the lower bound from Theorem~\ref{thm:bme} to the denominator and the upper bound to the numerator. But since $v$ is less than the optimal value of the sparsest cut as it is a relaxation of the problem, we have proved the first claim of the theorem as we simply return the cut which minimizes the above ratio.

  The final result of the theorem will follow from the claim that for all $i \in [r]$, we have either $T \subseteq S_i$ or $T \subseteq \bar{S}_i$ and we return whichever one does not contain $T$. To prove the claim, assume for the sake of contradiction that there exists $i \in [r]$ and $j,k \in T$ such that $j \in S_i$ and $k \in \bar{S}_i$. Then, we have:

  \begin{equation*}
    0 = \norm{f(k) - f(j)}_1 = \sum_{h = 1}^r \alpha_{S_h} \delta_{S_h} (j, k) \geq \alpha_{S_i} \delta_{S_i} (j, k) = \alpha_{S_i} > 0
  \end{equation*}

  which is a contradiction. This proves the claim and the second result of the theorem.
\end{proof}

\subsection{Extracting a Single Component}
For the purposes of our algorithm, we will consider a slightly different version of the sparsest cut problem. We begin by restating the definition of the expansion of a subset of the state space $S$:

\begin{definition}[Expansion]
  Given a matrix, $\bm{P}$, with non-negative entries, the expansion of a set $S$, denoted by $h_{\bm{P}}(S)$ is defined as:
  \begin{equation*}
    h_{\bm{P}} (S) = \frac{\sum_{i \in S, j \notin S} \bm{P}_{ij}}{\min (\abs{S}, \abs{\bar{S}})}
  \end{equation*}
\end{definition}

We will now re-state the definition of the Cheeger constant of a graph:
\begin{definition}[Cheeger Constant]
  The Cheeger Constant of a Markov Chain with transition matrix, $\bm{P}$, is the minimum expansion of any subset of the state space.

  \begin{equation*}
    \chg(\bm{P}) = \min_{S \subset [n]} h_{\bm{P}} (S)
  \end{equation*}
\end{definition}

\begin{algorithm}[H]
  \caption{Extract Component}
  \label{alg:ec}
  \begin{algorithmic}[1]
    \STATE \textbf{Input}: Transition Matrix $\bm{P}$, Extracted States $T$, Tolerance $\beta$

    \STATE $\S_0 \leftarrow \fc ([n], \bm{P}, T),\ t \leftarrow 0$ \label{ec:s0}
    \STATE $v_0 \leftarrow h_{\bm{P}} (\S_0)$

    \IF {$v_0 \geq \beta / 8$} 
      \STATE $\S_0 \leftarrow [n] \setminus T$
      \STATE $v_0 \leftarrow \abs{S_0}^{-1} \sum_{i,j \in S_0} \bm{P}_{ij}$
      \IF {$v_0 \leq 1 - \beta / 8$} \label{ec:afail}
        \STATE \textbf{Return: }False \label{ec:fret}
      \ENDIF
    \ENDIF \label{ec:intowhile}

    \WHILE {$\abs{\S_t} > 1$}
      \STATE $\Sp_t \leftarrow \fc (S_t, \bm{P}_{S_t}, \phi)$
      \STATE $v_t \leftarrow h_{\bm{P}_{S_t}} (\Sp_t)$

      \IF {$v_t \geq \beta / (8 \log n)$} \label{ec:stif}
        \STATE \textbf{break} \label{ec:stbreak}
      \ENDIF

      \STATE $u_{\Sp_t} \leftarrow \frac{\sum_{i, j \in \Sp_t} \bm{P}_{ij}}{\abs{\Sp_t}},\ u_{\bar{\Sp_t}} \leftarrow \frac{\sum_{i,j \in (S_t \setminus \Sp_t)} \bm{P}_{ij}}{\abs{S_t \setminus \Sp_t}}$
      \IF {$u_{\Sp_t} \leq u_{\bar{\Sp_t}}$}
        \STATE $S_{t + 1} \leftarrow \Sp_t$
      \ELSE
        \STATE $S_{t + 1} \leftarrow S_t \setminus \Sp_t$
      \ENDIF
      \STATE $t \leftarrow t + 1$
    \ENDWHILE

    \STATE \textbf{Return: } $S_t$
  \end{algorithmic}
\end{algorithm}

Here, we state a short lemma relating the expansion of a subset to its cut value.

\begin{lemma}
  \label{lem:expCut}
  For a matrix $\bm{P}$ with positive entries and a subset $S$, we have:

  \begin{equation*}
    \frac{n}{2} g_{\bm{P}} (S) \leq h_{\bm{P}} (S) \leq n g_{\bm{P}} (S)
  \end{equation*}

  Consequently, we have for the cut, $S^*$ returned by \fc{} when run with input $(\bm{P}, T)$:

  \begin{equation*}
    h_{\bm{P}} (S^*) \leq O(\log n) \min_{T \subseteq S \subset [n]} h_{\bm{P}} (S)
  \end{equation*}
\end{lemma}

\begin{proof}
  We first consider the case that $S \leq n / 2$. In this case, we have that $n / 2 \leq \abs{\bar{S}} \leq n$ and consequently:
  \begin{equation*}
    \frac{n}{2} g_{\bm{P}} (S) \leq h_{\bm{P}} (S) \leq n g_{\bm{P}} (S)
  \end{equation*}
  The alternate case is proved similarly.

  For the second claim of the lemma, we will again assume that $\abs{S^*} \leq n / 2$. Now, have from Theorem~\ref{thm:spccc} and the equation above:

  \begin{equation*}
    h_{\bm{P}} (S^*) \leq ng_{\bm{P}} (S^*) \leq n \cdot O(\log n) \min_{T \subseteq S \subset [n]} g_{\bm{P}} (S) \leq O(\log n) \min_{T \subseteq S \subset V} n \cdot \frac{2}{n} \cdot h_{\bm{P}} (S)
  \end{equation*}

  This proves the second claim of the lemma.
\end{proof}

The next lemma is the main result of the subsection concerning the performance of Algorithm~\ref{alg:ec}.

\begin{lemma}
  \label{lem:ep}
  Algorithm~\ref{alg:ec} runs in randomized polynomial time and either returns partition $S$ disjoint from $T$ satisfying:

  \begin{equation*}
    \text{Claim 1: } \frac{\sum_{i, j \in S} \bm{P}_{ij}}{\abs{S}} \geq 1 - \beta
    \qquad \text{Claim 2: } \forall R \subset \S\ \frac{\sum_{i \in R, j \in (\S \setminus R)} \bm{P}_{ij}}{\min(\abs{R}, \abs{\S \setminus R})} \geq \Omega \lprp{\frac{\beta}{\log^2 n}}
  \end{equation*}

  Or returns $False$ and certifies for all subsets $S \subset ([n] \setminus T)$, we have:

  \begin{equation*}
    \text{Claim 3: } h_{\bm{P}} (S) \geq \Omega \lprp{\frac{\beta}{\log n}}\qquad \text{Claim 4: } \frac{\sum_{i,j \in [n] \setminus T} \bm{P}_{ij}}{n - \abs{T}} \leq 1 - \frac{\beta}{8} 
  \end{equation*}
\end{lemma}

\begin{proof}
  We will first prove the third claim of the lemma. Let $\tilde{S}$ be the set returned in Line~\ref{ec:s0} of the algorithm. The only way the algorithm returns $False$ is if Line~\ref{ec:fret} is executed. Therefore, we have from the second claim of Lemma~\ref{lem:expCut} and the fact that Line~\ref{ec:fret} is executed:

  \begin{equation*}
    \frac{\beta}{8} \leq h_{\bm{P}} (\tilde{S}) \leq O(\log n) \min_{T \subseteq S \subset [n]} h_{\bm{P}} (S)
  \end{equation*}

  This proves the third claim of the lemma. The fourth claim of the lemma follows trivially from the fact that the if condition in Line~\ref{ec:afail} evaluates to true.

  Now, we will assume that the Algorithm is in the case where a set $S$ is returned. For the second claim of the lemma, the algorithm either returns a set containing a single element in which case, the claim is trivially true. In the alternate case, the break statement in Line~\ref{ec:stbreak} was executed and we have again from Lemma~\ref{lem:expCut}:

  \begin{equation*}
    \frac{\beta}{8 \log n} \leq h_{\bm{P}_S} (\Sp) \leq O(\log n) \min_{R \subset S} h_{\bm{P}_{S}} (S)
  \end{equation*}

  which implies the second claim of the Lemma.

  For the first claim of the lemma, assume that the inner loop runs for $K$ time steps. Now, consider the times $t_{0}, \dots, t_{k}$ defined as follows:
  \begin{equation*}
    t_{0} = 0, \qquad t_{k} = \min \{t \in [K]: \abs{S_{t}} \leq \abs{S_{t_{k - 1}}} / 2\}\ \forall k \in \{1, \dots, K - 1\},\qquad t_k = K
  \end{equation*}

  It is clear that $k$ is at most $\log n$. Now, we will prove the following claim:

  \begin{claim}
    \label{clm:ine}
    $\forall i \in \{0, \dots, k\}$, we have that $S_{t_i}$ satisfies:

    \begin{equation*}
      \frac{\sum_{i, j \in S_{t_i}} \bm{P}_{ij}}{\abs{S_{t_i}}} \geq 1 - \frac{\beta}{8} - \frac{i \beta}{4 \log n}
    \end{equation*}
  \end{claim}
  Instantiating Claim~\ref{clm:ine}, with $i = k$, proves the first claim of the Lemma by nothing that $k$ is at most $\log n$. Now, we will prove the claim via induction.\\
  \textbf{Base Case: } $i = 0$: The base case is true as the algorithm only proceeds beyond Line~\ref{ec:intowhile} if:
  \begin{equation*}
    \frac{\sum_{i \in S_{0},j \in \bar{S_{0}}} \bm{P}_{ij}}{\abs{S_{0}}} \leq \frac{\beta}{8} \implies \frac{\sum_{i,j \in S_{0}} \bm{P}_{ij}}{\abs{S_{0}}} \geq 1 - \frac{\beta}{8}
  \end{equation*}
  \textbf{Inductive Step: } Suppose that the claim is true for $l$, we will verify the claim for $l + 1$. Let $R_{m}$ denote the sets $(S_{m} \setminus S_{m + 1})$ for $m \in \{t_l, \dots, t_{l + 1} - 1\}$. Now, for $m \in \{t_l, \dots, t_{l + 1} - 1\}$:

  \begin{equation*}
    \sum_{i, j \in S_m} \bm{P}_{ij} = \sum_{i, j \in S_{m + 1}} \bm{P}_{ij} + \sum_{i, j \in R_{m}} \bm{P}_{ij} + \sum_{i \in S_{m + 1}, j \in R_{m}} \bm{P}_{ij}
  \end{equation*}

  Therefore, we have:

  \begin{equation*}
    \sum_{i, j \in S_{m + 1}} \bm{P}_{ij} + \sum_{i, j \in R_{m}} \bm{P}_{ij} = \sum_{i, j \in S_m} \bm{P}_{ij} - \sum_{i \in S_{m + 1}, j \in R_{m}} \bm{P}_{ij} \geq \sum_{i, j \in S_m} \bm{P}_{ij} - \frac{\beta}{8\log n} \abs{R_{m}}
  \end{equation*}

  where the last inequality follows because the algorithm will only proceed to step $m + 1$ if the condition in Line~\ref{ec:stif} of the Algorithm~\ref{alg:ec} fails. Rewriting the above inequality in terms of the quantities $u_{S^\prime_m}, u_{\bar{S}^\prime_m}$, we get:

  \begin{equation*}
    \abs{S^\prime_{m}}u_{S^\prime_m} + \abs{\bar{S}^\prime_m}u_{\bar{S}^\prime_m} \geq \sum_{i, j \in S_m} \bm{P}_{ij} - \frac{\beta}{8\log n} \abs{R_{m}}
  \end{equation*}

  From the above inequality, we may conclude by dividing both sides by $\abs{S_m}$ that (As the average of two numbers is always smaller than the larger number):
  \begin{multline*}
    \frac{\sum_{i, j \in S_{m + 1}} \bm{P}_{ij}}{\abs{S_{m + 1}}} \geq \frac{\sum_{i, j \in S_{m}} \bm{P}_{ij}}{\abs{S_{m}}} - \frac{\beta\cdot \abs{R_{m}}}{8\log n\cdot \abs{S_{m}}} \\
    \geq \frac{\sum_{i, j \in S_{m}} \bm{P}_{ij}}{\abs{S_{m}}} - \frac{\beta\cdot \abs{R_{m}}}{8\log n\cdot \abs{S_{t_{l}}} / 2} = \frac{\sum_{i, j \in S_{m}} \bm{P}_{ij}}{\abs{S_{m}}} - \frac{\beta\cdot \abs{R_{m}}}{4\log n\cdot \abs{S_{t_{l}}}}
  \end{multline*}
  where the second inequality follows from the fact that in the range of $m$, $\abs{S_m} \geq \abs{S_{t_l}} / 2$. By summing up the above inequality for $m$ ranging from $t_l$ to $t_{l + 1} - 1$, we get:

  \begin{equation*}
    \frac{\sum_{i, j \in S_{t_{l + 1}}} \bm{P}_{ij}}{\abs{S_{t_{l + 1}}}} \geq \frac{\sum_{i, j \in S_{t_l}} \bm{P}_{ij}}{\abs{S_{t_l}}} - \frac{\beta\cdot \sum_{m = t_l}^{t_l - 1}\abs{R_{m}}}{4\log n\cdot \abs{S_{t_{l}}}} \geq \frac{\sum_{i, j \in S_{t_l}} \bm{P}_{ij}}{\abs{S_{t_l}}} - \frac{\beta}{4\log n} \geq 1 - \frac{\beta}{8} - \frac{(l + 1) \beta}{4\log n}
  \end{equation*}
  where the second inequality follows from that fact that the $R_m$ are disjoint subsets of $S_m$ and the second inequality follows from the inductive hypothesis. This proves Claim~\ref{clm:ine} and as explained earlier, the claim implies the first claim of the lemma.
\end{proof}

\subsection{Partitioning the Markov Chain}
\label{prf:pg}
In this subsection, we will design an algorithm to partition the entire state space of the Markov Chain. Our graph partitioning algorithm is illustrated in Algorithm~\ref{alg:pg}. We recursively call Algorithm~\ref{alg:ec} and stop when no more components can be extracted from the state space. We then use the guarantees provided by Lemma~\ref{lem:ep} to prove Lemma~\ref{lem:pg}.

\begin{algorithm}[H]
  \caption{Partition Graph}
  \label{alg:pg}
  \begin{algorithmic}[1]
    \STATE \textbf{Input}: Transition Matrix $\bm{P}$, Tolerance $\beta$

    \STATE $\mathcal{S} \leftarrow \{\}$
    \STATE $t \leftarrow 0$
    \STATE $T_t \leftarrow \phi$
    \STATE $S_t \leftarrow \text{Extract Component}(\bm{P}, T_t, \beta)$

    \WHILE {$S_t \neq False$}
      \STATE $\mathcal{S} \leftarrow \mathcal{S} \cup \{S_t\}$
      \STATE $T_{t + 1} \leftarrow T_{t} \cup S_{t}$
      \STATE $t \leftarrow t + 1$
      \STATE $S_t \leftarrow \text{Extract Component}(\bm{P}, T_t, \beta)$
    \ENDWHILE

    \STATE \textbf{Return: } $(\mathcal{S}, [n] \setminus T_t)$
  \end{algorithmic}
\end{algorithm}

  We will now proceed with the proof of Lemma~\ref{lem:pg}. We first note that $T_t \neq \phi$ at the end of the algorithm as this would violate Claim 4 of Lemma~\ref{lem:ep}. Now, we have by induction that $T_0 = \phi$ and $T_t = \bigcup_{i = 0}^{t - 1} S_i$. We also have by Lemma~\ref{lem:ep}, that $S_t$ is disjoint with $T_{t}$ and is therefore disjoint with $S_{0}, \dots, S_{t - 1}$. This shows that the subsets in $\mathcal{S}$ are disjoint. Suppose the algorithm terminates with $t = l$, note that $T_l = \bigcup_{S \in \mathcal{S}} S$ and consequently $T = [n] \setminus T_l$ and this proves the final claim of the lemma that the subsets in $\mathcal{S}$ along with $T$ form a partition of $[n]$.

  For the first two claims of the lemma, we have for all $S \in \mathcal{S}$, $S$ is returned by Algorithm~\ref{alg:ec} and the first two claims follow from the first two claims of Lemma~\ref{lem:ep}.

  We now prove the third claim of the lemma. We first note that if $T \neq \phi$, then from Claim 4 of Lemma~\ref{lem:ep} for $T$ and Claim 1 of Lemma~\ref{lem:ep} for each $S \in \mathcal{S}$:
  \begin{equation*}
    \frac{\beta}{8} \cdot \abs{T} \leq \sum_{i \in T, j \in \bar{T}} \bm{P}_{ij} = \sum_{S \in \mathcal{S}} \sum_{j \in S, i \in T} \bm{P}_{ij} \leq \sum_{S \in \mathcal{S}} \beta \abs{S} \implies \abs{T} \leq \frac{8n}{9}
  \end{equation*}

  Now, let any $R \subset T$. In the case that $\abs{R} \leq n / 2$, Claim 3 follows from Claim 3 of Lemma~\ref{lem:ep}. For $\abs{R} \geq n / 2$, note that $\abs{R} \leq 8n / 9$. Therefore, we have from Claim 3 of Lemma~\ref{lem:ep}:
  \begin{equation*}
    \Omega \lprp{\frac{\beta}{\log n}} \leq h_{\bm{P}} (R) = \frac{\sum_{i \in R, j \in \bar{R}} \bm{P}_{ij}}{\abs{\bar{R}}} \leq 9 \frac{\sum_{i \in R, j \in \bar{R}} \bm{P}_{ij}}{n} \leq 9 \frac{\sum_{i \in R, j \in \bar{R}} \bm{P}_{ij}}{\abs{R}}
  \end{equation*}
  and Claim 3 follows.

\qed

\section{Markov Chain Properties}
\label{sec:mcprop}

Our first lemma concerns bounding the amount of time the trajectory of the Markov Chain spends in the component $T$. The proof of our lemma follows along the lines of Lemma~3.3 in \cite{sinclair1989approximate}. In our lemma, we bound the first eigenvalue of a sub-matrix of the transition matrix whereas in \cite{sinclair1989approximate}, the same techniques are used to bound the second eigenvalue of the whole transition matrix.

\begin{lemma}
  \label{lem:tnb}
  Let $\bm{P}$ be the transition matrix of a symmetric markov chain. Let $T \subset [n]$ satisfy:

  \begin{equation*}
    \forall R \subseteq T\ \frac{\sum_{i \in R, j \in \bar{R}} \bm{P}_{ij}}{\abs{R}} \geq \alpha
  \end{equation*}

  Then, $\bm{P}_T$ has the following bound on its spectral norm:

  \begin{equation*}
    \norm{\bm{P}_T} \leq 1 - \frac{\alpha^2}{2}
  \end{equation*}
\end{lemma}

\begin{proof}
  Since $\bm{P}_T$ is symmetric and positive, its top eigenvalue, denoted by $\lambda$, is the same as its top singular value. Now, let $\abs{T} = m$ and let $\bm{u} \in \mb{R}^m$ be the eigenvector associated with the top eigenvalue. We will now suppose without loss of generality that $\bm{u}_1 \geq \bm{u}_2 \dots{} \bm{u}_{m-1} \geq \bm{u}_m \geq 0$ from the Perron-Frobenius Theorem. Now, we have:

  \begin{equation*}
    \bm{P}_T \bm{u} = \lambda \bm{u} \implies (\bm{I} - \bm{P}_T)\bm{u} = (1 - \lambda)\bm{u} \implies (1 - \lambda) = \bm{u}^\top (\bm{I} - \bm{P}_T) \bm{u}
  \end{equation*}

  We will now extend the vector $\bm{u}$ to a vector  $\bm{v} \in \mb{R}^{m + 1}$. Such that $\bm{v}_i = \bm{u}_i$ for all $i \in \{1, \dots, m\}$ and $\bm{v}_{m + 1} = 0$. Similarly, we extend $\bm{P}_T$ to an matrix $\bm{R} \in \mb{R}^{(m + 1) \times (m + 1)}$. Such that:
  
  \begin{equation*}
    \bm{R}_{ij} = \begin{cases}
                    (\bm{P}_T)_{ij}, &\text{for $i,j \in [m]$} \\
                    1 - \sum_{k \in [m]} (\bm{P}_T)_{ik}, &\text{for $i \in [m],\ j = m+1$} \\
                    1 - \sum_{k \in [m]} (\bm{P}_T)_{kj}, &\text{for $i = m+1,\ j \in [m]$} \\
                    0, &\text{otherwise}
                  \end{cases}  
  \end{equation*}

  Notice that $\bm{u}^\top (\bm{I} - \bm{P}_T) \bm{u} = \bm{v}^\top (\bm{I} - \bm{R}) \bm{v}$. Now, we expand the right hand side as follows:

  \begin{equation}
  \label{eqn:qfbnd}
    \bm{v}^\top (\bm{I} - \bm{R}) \bm{v} = \sum_{i = 1}^{m + 1}  v_i^2 - \sum_{i, j} (\bm{R})_{ij} v_i v_j = \sum_{i = 1}^{m + 1} (1 - \bm{R}_{ii})v_i^2 - 2 \sum_{i < j} \bm{R}_{ij} v_i v_j = \sum_{i < j} \bm{R}_{ij} (v_i - v_j)^2 
  \end{equation}

  Now, consider the equation:
  \begin{equation}
  \label{eqn:csbnd}
    \sum_{i < j} \bm{R}_{ij} (v_i + v_j)^2 \leq 2\sum_{i < j} \bm{R}_{ij} (v_i^2 + v_j^2) \leq 2\sum_{i, j \in [m + 1]} \bm{R}_{ij} v_i^2 = 2
  \end{equation}

  Now, we get from Equations~\ref{eqn:qfbnd} and \ref{eqn:csbnd}:
  \begin{equation}
  \label{eqn:pcsb}
    \bm{v}^\top (\bm{I} - \bm{R}) \bm{v} \geq \sum_{i < j} \bm{R}_{ij} (v_i - v_j)^2 \cdot \frac{\sum_{i < j} \bm{R}_{ij} (v_i + v_j)^2}{2} \geq \frac{1}{2} \cdot \lprp{\sum_{i < j} \bm{R}_{ij} (v_i^2 - v_j^2)}^2
  \end{equation}
  where the last inequality follows from Cauchy-Schwarz. Now, we will bound the term in the parenthesis in the final expression on the right hand side:

  \begin{align*}
    \sum_{i < j} \bm{R}_{ij} (v_i^2 - v_j^2) &= \sum_{i < j} \bm{R}_{ij} \sum_{k = i}^{j - 1} (v_{k}^2 - v_{k + 1}^2) = \sum_{k = 1}^{m} (v_k^2 - v_{k + 1}^2) \sum_{j > k, i \leq k} \bm{R}_{ij} \\
    & \geq \sum_{k = 1}^{m} (v_k^2 - v_{k + 1}^2) \alpha k = \alpha \sum_{j = 1}^m \sum_{k = j}^m (v_k^2 - v_{k + 1}^2) = \alpha \sum_{j = 1}^m v_j^2 = \alpha
  \end{align*}

  where the first inequality follows from the assumption on $\bm{P}$ and $T$ and the subsequent equality from the fact that $v_{m + 1} = 0$. Substituting the inequality in Equation~\ref{eqn:pcsb}, we get the desired result.
\end{proof}
\section{Deferred Proofs from Section~\ref{sec:mpr}}
\label{app:dptst}

\subsection{Proof of Lemma~\ref{lem:helbnd}}
\label{prf:helbnd}

  Let $l = \abs{T}$ and $v = \frac{1}{\sqrt{l}} \bm{1}_{T}$. Now, we consider two cases:

  \textbf{Case 1: } First, we consider the case where $\sum_{i, j \in T} \bm{Q}_{ij} \geq (1 - 5\epsilon / 16)l$. In this case, we have by the definition of $\distt$:
  \begin{align*}
    \dhel^2 \lprp{\distt(T, \bm{P}), \distt(T, \bm{Q})} &= \frac{1}{2} \lprp{\sum_{i \in T, j \in T} \frac{1}{l} \lprp{\sqrt{\bm{P}_{ij}} - \sqrt{\bm{Q}_{ij}}}^2 + \lprp{\sqrt{\distt(T, \bm{P}) (\nue)} - \sqrt{\distt(T, \bm{Q}) (\nue)}}^2} \\
    &\geq \frac{1}{2l} \sum_{i,j \in T} \lprp{\sqrt{\bm{P}_{ij}} - \sqrt{\bm{Q}_{ij}}}^2 \geq 1 - \frac{3 \epsilon}{16} - \sum_{i,j \in T} \frac{\sqrt{\bm{P}_{ij}\bm{Q}_{ij}}}{l} \\
    &= 1 - \frac{3\epsilon}{16} - v^\top \sq (\bm{Q}, \bm{P}) v \geq \frac{\epsilon}{2}
  \end{align*}

  where the second inequality is from our assumption on $T$ and $\bm{P}$ and the final inequality is from our definition of $\distt(\bm{P}, \bm{Q})$.

  \textbf{Case 2: } For the alternative case, we have $s = \sum_{i, j \in T} \bm{Q}_{ij} \leq (1 - 5\epsilon / 16)l$. In this case, we have from the definition of $\dtv$:
  \begin{equation*}
    \dtv \lprp{\distt(T, \bm{P}), \distt(T, \bm{Q})} \geq \frac{1}{l} \sum_{i, j \in T} \bm{P}_{ij} - \bm{Q}_{ij} \geq \frac{\epsilon}{4}
  \end{equation*}
  Therefore, we have from the relationship between the Hellinger distance and Total Variation distance in Definition~\ref{def:dheltv}:
  \begin{equation*}
    \dhel^2 \lprp{\distt(T, \bm{P}), \distt(T, \bm{Q})} \geq \frac{\epsilon^2}{32}
  \end{equation*}

\qed

\subsection{Proof of Lemma~\ref{lem:tli}}
\label{prf:tli}

  We begin by partitioning the word into $l / k$ blocks of length $k = \frac{2\log n}{\alpha^2}$ and let $Y_j$ denote the random variable denoting whether there is an element $X_k \notin T$ in the $j^{th}$ block. That is:

  \begin{equation*}
    Y_j = \bm{1} \{\exists i \in [(j - 1)k + 1, jk]: X_i \notin T\}
  \end{equation*}
  We will now prove bound $\mb{P} \{Y_j = 1 | X_1, \dots, X_{(j - 1)k}\}$. We will consider two cases:

  \textbf{Case 1: } $X_{j(k - 1) + 1} \notin T$. In this case, we have $\mb{P} \{Y_j = 1 | X_1, \dots, X_{(j - 1)k}\} = 1$.

  \textbf{Case 2: } In this case assume $X_{(j - 1)k + 1} = x \in T$. Here, we have from the property of the Markov chain that:
  \begin{equation*}
    \mb{P} \{Y_j = 0 | X_1, \dots, X_{(j - 1)k}, X_{(j - 1)k + 1} = x\} = e_{x}^\top \bm{P}_T^{k - 1} \bm{1} \leq \sqrt{n} \norm{\bm{P}_T}^{k - 1} \leq \frac{1}{2}
  \end{equation*}
  where the first inequality follows form Cauchy-Schwarz and the second follows from Lemma~\ref{lem:tnb}.

  Therefore, by combining the two cases above we have $\mb{P} \{Y_j = 1 | X_1, \dots, X_{(j - 1)k}\} \geq 0.5$ and we get:

  \begin{equation*}
    \mb{P} \lbrb{\sum_{i = 1}^l \bm{1}\{X_i \notin T\} \geq \frac{l}{8 \log n \alpha^2}} \geq \mb{P} \lbrb{\sum_{i = 1}^{l / k} Y_{i} \geq \frac{l}{4k}} \geq 1 - \delta
  \end{equation*}

  via an application of Hoeffding's inequality (See, for example, \cite{boucheron2013concentration}) and using our bound on $l / k$.

\qed

\subsection{Proof of Lemma~\ref{lem:mxk}}
\label{prf:mxk}

  We start by first fixing a particular element $i \in [k]$. Now, we have:
  \begin{equation*}
    \mb{E} [v_i] = \frac{m}{k}
  \end{equation*}
  Therefore, we have by an application of Theorem~1.1 in \cite{dubhashi_panconesi_2009} that:
  \begin{equation*}
    \mb{P} \lbrb{v_i \geq 2 \frac{m}{k}} \leq \exp \lprp{- \frac{m}{3k}} \leq \exp \lprp{-3\log \frac{n}{\epsilon}} \leq \lprp{\frac{\epsilon}{n}}^3
  \end{equation*}
  Finally, we get via an application of the union bound:
  \begin{equation*}
    \mb{P} \lbrb{\max_{i \in [k]} v_i \geq 2 \frac{m}{k}} \leq k \lprp{\frac{\epsilon}{n}}^3 \leq \frac{\epsilon}{n^2}
  \end{equation*}

\qed

\subsection{Proof of Lemma~\ref{lem:submc}}
\label{prf:submc}

  Since the chain $\bm{Y}$ is irreducible, we have that $\bm{X}$ is defined almost surely. Now, we will prove that $q_{ij} = \mb{P} \{X_{k + 1} = j | X_{k} = i\}$ is independent of $k$. We will do this by showing that $\bm{P} [X_{k + 1} = j| X_{k} = i, \tau_k = l]$ is independent of $l$ and $k$ as:
  \begin{align*}
    \mb{P} \{X_{k + 1} = j | X_{k} = i\} &= \sum_{l = 1}^\infty \mb{P} \{X_{k + 1} = j, \tau_k = l| X_{k} = i\} \\ 
    &= \sum_{l = 1}^\infty \mb{P} \{\tau_k = l| X_{k} = i\} \mb{P} \{X_{k + 1} = j | \tau_k = l, X_{k} = i\}
  \end{align*}

  Now, we define $\mathcal{P}_k$ to be sequences of states of length $k$ that begin with $i$ and end with $j$ but the elements in between are not in $T$. That is, if $(i_1, i_2, \dots, i_k) \in \mathcal{P}_k$, then we have $i_1 = i, i_k = j$ and $i_l \notin T,\ \forall l \in \{2, \dots, k-1\}$. Therefore, we get by the markov property of $\bm{Y}$ and the definition of $\bm{X}$:
  \begin{align*}
    \mb{P} \{X_{k + 1} = j | \tau_k = l, X_{k} = i\} &= \mb{P} \{Y_{\tau_{k + 1}} = j | Y_l = i\} = \sum_{m = l+1}^\infty \mb{P} \{\tau_{k + 1} = m, Y_{m} = j | Y_l = i\} \\
    &= \sum_{m = 2}^\infty \mb{P} \{\tau_{2} = m, Y_{m} = j | Y_1 = i\} = \sum_{r = 2}^\infty \sum_{\bm{i} \in \mathcal{P}_r} \prod_{s = 1}^{r - 1} \bm{P}_{i_si_{s+1}} \\
    &= \bm{P}_{ij} + \bm{e}_{i}^\top \lprp{\sum_{t = 1}^\infty \bm{P}_{T, \bar{T}} \bm{P_{\bar{T}}}^{t} \bm{P}_{\bar{T}, T}} \bm{e}_j
  \end{align*}

  This is independent of $k$ and therefore, the process $\bm{X}$ is a markov process and the claim about the transition matrix follows from the above expression as, we have for all $i,j \in T$ and $k \in \mb{N}$ $\mb{P}[X_{k + 1} = j | X_k = i] = \bm{Q}_{ij}$.

\qed

\subsection{Proof of Corollary~\ref{cor:expsmc}}
\label{prf:expsmc}

  The corollary is immediate as $\forall S \subset T, \abs{S} \leq \abs{T} / 2$:
  \begin{equation*}
    h_{\bm{Q}} (S) = \frac{\sum_{i \in S, j \in T \setminus S} \bm{Q}_{ij}}{\abs{S}} \geq \frac{\sum_{i \in S, j \in T \setminus S} \bm{P}_{ij}}{\abs{S}} \geq \alpha
  \end{equation*}
  where the last bound follows from the fact that $\bm{Q}_{ij} \geq \bm{P}_{ij}$ from Lemma~\ref{lem:submc}.

\qed

\subsection{Proof of Lemma~\ref{lem:htb}}
\label{prf:htb}

  To start, consider the markov chain with transition matrix $\bm{Q} = 0.5 (\bm{P} + \bm{P}^2)$. Given a trajectory of length $2l$ from the transition matrix $\bm{P}$, it is easy to simulate a trajectory of length $l$ from $\bm{Q}$ by simply taking the next element in the trajectory with probability $0.5$ and skipping an element with probability $0.5$. It follows that $\text{HitT}(\bm{P})$ is upper bounded by $2\hitt(\bm{Q})$.

  Now, let $1 = \lambda_1 \geq \lambda_2 \geq \dots \geq \lambda_n \geq -1$ be the eigenvalues of $\bm{P}$ and let $v_1, \dots, v_n$ be the eigenvectors. Note that we can take $v_1$ to be the vector $(1 / \sqrt{n}, 1 / \sqrt{n}, \dots, 1 / \sqrt{n})$ (The unit vector in the direction of the stationary distribution). Now, let $\pi$ be any distribution over the states $[n]$. Now, we have $\inp{\pi}{v_1} = 1 / \sqrt{n}$. And furthermore, we have $\forall i \in [n]$, $\inp{v_i}{\pi} \leq 1$. Now, note that since $\bm{P}$ and $\bm{Q}$ have the same set of eigenvectors $v_1, \dots, v_n$ and the corresponding eigenvalues for $\bm{Q}$ are $0.5 (\lambda_1 + \lambda_1^2), \dots, 0.5(\lambda_n + \lambda_n^2)$. Now, let $1 = \sigma_1 > \sigma_2 \dots \geq \sigma_n$ be the eigenvalues of $\bm{Q}$ with eigenvectors $v_1, u_2, \dots, u_n$. We have from the previous Lemma~\ref{lem:sinc} that:

  \begin{equation*}
    \abs{\sigma_i} \leq 1 - \frac{\alpha^2}{2}
  \end{equation*}

  as when $\lambda \leq 0$, the maximum absolute value of $0.5 (\lambda + \lambda^2)$ is $1 / 8$. Now, let $\pi_0$ be any starting distribution over states, then the distribution over the states at time $t$, $\pi_t$,  is $\pi_0 \bm{Q}^t$ and $\pi^*$ be the stationary distribution. Therefore, we get:

  \begin{equation*}
    \norm{\pi_t - \pi^*} = \norm*{\frac{1}{\sqrt{n}} v_1 - \pi^* + \sum_{i = 2}^n \sigma_i^t \inp{u_i}{\pi_0}u_i} \leq \sum_{i = 2}^n \sigma_i^t \abs{\inp{u_i}{\pi_0}} \leq n \lprp{1 - \frac{\alpha^2}{2}}^t \leq n\exp\lprp{- \frac{\alpha^2}{2}\cdot t}
  \end{equation*}

  Therefore, we have at $t^* = 4\log(10n) / \alpha^2$, we have:

  \begin{equation}
    \label{eqn:pbnd}
    \norm{\pi_t - \pi^*} \leq \frac{1}{4n}
  \end{equation}

  Therefore, we have by Equation~\ref{eqn:pbnd}:

  \begin{equation*}
    \hitt(\bm{Q}) \leq 4 \frac{\log (10n)}{\alpha^2} \cdot \frac{3}{4n} + \lprp{1 - \frac{3}{4n}} \lprp{4 \frac{\log (10n)}{\alpha^2} + \hitt(\bm{Q})}
  \end{equation*}

  By rearranging the above inequality, we get:

  \begin{equation*}
    \hitt(\bm{Q}) \leq 10 \frac{\log (10n)}{\alpha^2} \implies \hitt(\bm{Q}) \leq \ot \lprp{\frac{n}{\alpha^2}} \implies \hitt(\bm{P}) \leq \ot \lprp{\frac{n}{\alpha^2}}
  \end{equation*}

\qed

\end{document}